\documentclass{amsart}

\usepackage{amsmath,amsthm,amssymb,mathrsfs}
\usepackage{xypic}
\usepackage{hyperref}
\usepackage{color}
\usepackage{cite}
\usepackage{comment}
\theoremstyle{theorem}

\newtheorem{thm}{Theorem}[section]

\newtheorem{prop}[thm]{Proposition}
\newtheorem{cor}[thm]{Corollary}

\newtheorem*{thm*}{Theorem}
\newtheorem*{prop*}{Proposition}
\newtheorem*{lem*}{Lemma}
\newtheorem*{cor*}{Corollary}

\newtheorem{mainthm}{Theorem}
\newtheorem{maincor}[mainthm]{Corollary}

\theoremstyle{remark}
\newtheorem{remi}[thm]{Remark}
\newtheorem{ex}[thm]{Example}

\newcommand{\g}{\mathfrak g}

\newcommand{\pM}{\pi_{M}^{*}}
\newcommand{\pN}{\pi_{N}^{*}}

\begin{document}

\title{The Euler characteristic of a transitive Lie algebroid}
\author{James Waldron}

\maketitle

\begin{abstract}
We apply the Atiyah-Singer index theorem and tensor products of elliptic complexes to the cohomology of transitive Lie algebroids. We prove that the Euler characteristic of a representation of a transitive Lie algebroid $A$ over a compact manifold $M$ vanishes unless $A=TM$, and prove a general K\"{u}nneth formula. As applications we give a short proof of a vanishing result for the Euler characteristic of a principal bundle calculated using invariant differential forms, and show that the cohomology of certain Lie algebroids are exterior algebras. The latter result can be seen as a generalization of Hopf's theorem regarding the cohomology of compact Lie groups.

\end{abstract}

\section{Introduction}

\subsection{Euler characteristics of Lie algebras}

Let $\g$ be a finite dimensional Lie algebra. The Lie algebra cohomology $H^{\bullet}\left(\g\right)$ is a finite dimensional graded vector space concentrated in degrees $0\le p\le \mathrm{dim}\;\g$. This permits one to define the \emph{Euler characteristic} of $\g$ as the alternating sum
\[
\chi\left(\g\right) := \sum_{p=0}^{\mathrm{dim}\;\g} \left(-1\right)^{p} \mathrm{dim}\;H^{p}\left(\g\right).
\]
The motivation for this paper is the following Theorem and its proof.
\begin{thm}
\label{thm: Goldberg}
\emph{(Goldberg 1955,\cite{Goldberg55}).}
If $\g$ is a non-zero finite dimensional Lie algebra then $\chi\left(\g\right)=0$.
\end{thm}
Note that if $\g=0$ then $\chi\left(\g\right)=1$. This result has an interesting history, having been proven earlier by Chevalley \& Eilenberg \cite{ChevalleyE48} for the classical Lie algebras using results on the structure of simple Lie groups; see \cite{Zusmanovich11} for a discussion. The proof given in \cite{Goldberg55} is purely algebraic and works over any field: one applies the `Euler-Poincar\'{e} principle' to the Chevalley-Eilenberg complex $\wedge^{\bullet}\g^{*}$ and uses the fact that the alternating sum of the binomial coefficients vanishes. In particular, the proof does not involve the differential but only the vector spaces appearing in the complex. 

For an action of a Lie group $G$ on a manifold $M$ we denote by $H_{\mathrm{dR},G}^{\bullet}\left(M\right)$ the cohomology of the complex of $G$-invariant differential forms. Following \cite{TangYZ13}, if the cohomology groups $H_{\mathrm{dR},G}^{p}\left(M\right)$ are finite dimensional then we define the \emph{Euler characteristic of the $G$-action on $M$} by
\[
\chi\left(M,G\right) := \sum_{p=0}^{\mathrm{dim}\;M}\mathrm{dim}\left(-1\right)^{p}\;H_{\mathrm{dR},G}^{p}\left(M\right)
\]
and denote the standard Euler characteristic by $\chi\left(M\right)$.
If $G$ is a Lie group with Lie algebra $\g$ then Theorem \ref{thm: Goldberg} and the isomorphism $H_{\mathrm{dR},G}^{\bullet}\left(G\right) \cong H^{\bullet}\left(\g\right)$ proves the following Corollary.
\begin{cor}
\label{cor: Goldberg}
Let $G$ be a positive dimensional Lie group acting on itself via the right action.
\begin{enumerate}
\item $\chi\left(G,G\right)=0$.
\item $\chi\left(G\right) = 0$ if $G$ is compact. 
\end{enumerate}
\end{cor}
The second statement is well known and is usually proven using topological arguments, e.g.~ via the Lefschetz trace formula, the Poincar\'{e}-Hopf index theorem, or the vanishing of the Euler class of a parallelizable manifold. Theorem \ref{thm: Goldberg} can be seen as providing a purely algebraic explanation of this fact.

\subsection{Transitive Lie algebroids}
\label{sec: TLA}

Our first main result is a generalization of Theorem \ref{thm: Goldberg} and its proof to the case of \emph{transitive Lie algebroids}, and of Corollary \ref{cor: Goldberg} to principal bundles.  A transitive Lie algebroid over a smooth manifold $M$ is a smooth vector bundle $A$ over $M$ equipped with a surjective vector bundle morphism $a:A \to TM$ and a Lie bracket on $\Gamma\left(A\right)$ satisfying an analogue of the Leibniz rule. Standard examples include finite dimensional real or complex Lie algebras ($M=\mathrm{pt}$), the tangent bundle $TM$, and the \emph{Atiyah algebroid} $TP/G$ of a principal $G$-bundle $P\to M$ for $G$ a Lie group. There is a notion of \emph{representation} of a Lie algebroid on a vector bundle $E$, to which there are associated cohomology groups $H^{p}\left(A,E\right)$. See section \ref{sec: bg} for the precise definitions.

\subsection{Main results}

If $E$ is a representation of a transitive Lie algebroid $A$ and the cohomology groups $H^{p}\left(A,E\right)$ are finite dimensional then we define the \emph{Euler characteristic} of $E$ as 
\[
\chi\left(A,E\right) := \sum_{p=0}^{\mathrm{rank}\;A}\left(-1\right)^{p}\mathrm{dim}\; H^{p}\left(A,E\right).
\]
We write $H^{p}\left(A\right)$ and  $\chi\left(A\right)$ for $E=M\times \mathbb{R}$ the standard representation.
\begin{mainthm}
\label{thm: euler}

Let $A$ be a real or complex transitive Lie algebroid over a connected compact manifold $M$, $L=\mathrm{Ker}\;a$, and $E$ a representation of $A$. 
\[
\chi\left(A,E\right) = \begin{cases}  \mathrm{rank}\;E \cdot \chi\left(M\right) & \text{if } \; L=0 \\
0 &  \text{otherwise. } 
\end{cases}
\]
\end{mainthm}

The proof of Theorem \ref{thm: euler} uses the cohomological form of the Atiyah-Singer index theorem \cite{AtiyahS68b} applied to the elliptic complex $\Gamma\left(E\otimes \wedge^{\bullet}A^{*}\right)$ computing $H^{\bullet}\left(A,E\right)$. We show that the integrand in the index theorem is equal to the Euler class of $M$ multiplied by the integer 
\[
 \left(\sum_{p=0}^{\mathrm{rank}\;L}\left(-1\right)^{p} \mathrm{rank}\;\wedge^{p}L^{*}\right) \mathrm{rank}\;E,
\]
which vanishes whenever $L\ne 0$ for the same reason as in the proof of Theorem \ref{thm: Goldberg}. Specialising to the case $M=\mathrm{pt}$ recovers Goldberg's theorem, and to the case $A=TM$ the computation of the Euler characteristic of a local system.

\begin{maincor}
\label{cor: euler}
Let $G$ be a positive dimensional Lie group and $P$ a principal $G$-bundle over a compact manifold $M$.
\begin{enumerate}
\item The cohomology groups $H_{\mathrm{dR},G}^{p}\left(P\right)$ are finite dimensional.
\item $\chi\left(P,G\right)=0$.
\item $\chi\left(P\right)=0$ if $G$ is compact.
\end{enumerate}
The same results hold if $\Omega^{\bullet}\left(P\right)^{G}$ is replaced by $\Omega^{\bullet}\left(P,V\right)^{G}$ for $V$ a non-zero finite dimensional real or complex representation of $G$.
\end{maincor}

Corollary \ref{cor: euler} is a special case of the following more general result which is part of Theorems 1.1 and 4.1 of \cite{TangYZ13}. The proofs are independent. We understand that this result was already known to the authors of loc.~ cit.~
\begin{thm}
\label{thm: TYZ}
\emph{(Tang, Yao \& Zhang, 2013, \cite{TangYZ13}.)}
Let $M$ be a manifold on which a Lie group $G$ acts properly and cocompactly. 
\begin{enumerate}
\item The cohomology groups $H_{\mathrm{dR},G}^{p}\left(M\right)$ are finite dimensional.
\item If the dimension of $M$ is odd or there exists a nowhere vanishing $G$-invariant vector field on $M$ then $\chi\left(M,G\right)=0$.
\end{enumerate}
\end{thm}

Corollary \ref{cor: euler} is proved by applying Theorem \ref{thm: euler} to the Atiyah algebroid of $P$. This result can also be restated in the following equivalent way: if $G$, $P$ and $M$ are as in the statement and $M$ is considered as a trivial $G$-space then
\[
\chi\left(P,G\right)= \chi\left(M,G\right) \cdot \chi\left(G,G\right),
\]
which reduces to Serre's identity 
\[
\chi\left(P\right)= \chi\left(M\right) \cdot \chi\left(G\right)
\]
\cite{Serre51} if $G$ is compact.

Our second main result is a K\"{u}nneth theorem for transitive Lie algebroids. Let $A$ resp.~ $B$ be a transitive Lie algebroid over a compact manifold $M$ resp.~ $N$ and $E$ resp.~ $F$ be a representation of $A$ resp.~ $B$. The \emph{product Lie algebroid} $A\times B = A\boxplus B := \pi_{M}^{*}A \oplus \pi_{N}^{*}B$ is a Lie algebroid over $M\times N$ and the vector bundle $E\boxtimes F := \pi_{M}^{*}E \otimes \pi_{N}^{*}F$ is a representation of $A\boxplus B$ in a natural way. For the precise details see the proof of Theorem \ref{thm: kunneth} in section \ref{sec: kunnethproof}. 
\begin{mainthm}
\label{thm: kunneth}
With the notation as above there is an isomorphism of graded vector spaces
\[
H^{\bullet}\left(A \times B,E\boxtimes F\right) 
\cong H^{\bullet}\left(A,E\right) \otimes H^{\bullet}\left(B,F\right)
\]
which is an isomorphism of graded algebras if $E=M\times \mathbb{R}$ and $F=M\times \mathbb{R}$ are the standard representations.
\end{mainthm}
The proof of Theorem \ref{thm: kunneth} is an application of the K\"{u}nneth theorem for elliptic complexes stated in \cite{AtiyahB67}; see also Theorem 1.3 in \cite{Tarkhanov86} and section 1.4.3 in \cite{TarkhanovBook}. Specialising to the case $M=\mathrm{pt}$ recovers the K\"{u}nneth formula for Lie algebras, and to the case $A=TM$ and $B=TN$ the K\"{u}nneth theorem for de Rham cohomology with local coefficients. Theorem \ref{thm: kunneth} answers a question posed by Kubarski in \cite{Kubarski02}, where the result is proven for the case where $A=TM$, $N=\mathrm{pt}$ and $E$ and $F$ are the standard representations.

\begin{maincor}
\label{cor: kunneth}
Let $G$ and $H$ be Lie groups and $P$ resp.~ $Q$ be a principal $G$ resp.~ $H$ bundle with $P/G$ and $Q/H$ compact. There is an isomorphism of graded vector spaces
\[
H_{\mathrm{dR},G}^{\bullet}\left(P\right) \otimes H_{\mathrm{dR},H}^{\bullet}\left(Q\right) \cong H_{\mathrm{dR},G\times H}^{\bullet}\left(P \times Q \right)
\]
where $G\times H$ acts on $P\times Q$ via the diagonal action.
\end{maincor}

If $A$ is a Lie algebroid then by a \emph{compatible smooth $H$-space structure} we shall mean a Lie algebroid morphism $H:A\times A \to A$ for which there exists an element $e \in A$, called a \emph{unit for $H$} that is contained in the zero section and satisfies $H\left(e,x\right)=H\left(x,e\right)=x$ for all $x \in A$. In particular, $H$ makes $A$ into an $H$-space in the sense of topology \cite{HatcherBook}. Note that if $H$ is in fact associative and has inverses then $A$ is an example of an `$\mathcal{LA}$-groupoid' \cite{Mackenzie92}.
\begin{maincor}
\label{cor: hopf}
Suppose that $A$ is a transitive Lie algebroid over a connected compact manifold $M$ and $A$ is equipped with a compatible smooth $H$-space structure $H:A\times A \to A$. Then $H^{\bullet}\left(A\right)$ is isomorphic to a graded exterior algebra with odd degree generators and carries the structure of a graded Hopf algebra if $H$ is associative.
\end{maincor}

If $A=TG$ for $G$ a compact Lie group and $H$ equal to the derivative of the multiplication of $G$ then $H^{\bullet}\left(A\right) = H_\mathrm{dR}^{\bullet}\left(G\right)$ and Corollary \ref{cor: hopf} reduces to the theorem of Hopf on the cohomology of compact Lie groups \cite{Hopf41}. We show in section \ref{sec: examplesH} that if $A$ is transitive then the existence of an $H$-structure is fairly restrictive, in particular the fibres of $L=\mathrm{Ker}\;a$ are necessarily abelian.

\subsection{Relation to existing work}

Theorem \ref{thm: euler} is an extension of the following two Theorems which compute the Euler characteristic $\chi\left(A\right)$ of the standard representation under additional assumptions on $M$ and $A$. 

\begin{thm}
\label{thm: IKV}
\emph{(Itskov, Karasev \& Vorobjev 1998,  \cite{ItskovKV98}, Corollary 4.11.)} If $M$ is simply connected then $\chi\left(A\right) = \chi\left(\g\right) \chi\left(M\right)$, where $\g = \mathrm{Ker}\; a_{x}$ for some $x \in M$.
\end{thm}

\begin{thm}
\label{thm: K}
\emph{(Kubarski 2002, \cite{Kubarski02}, Proposition 7.6.)}
If $A$ is transitive unimodular invariantly oriented, $M$ is oriented and $\mathrm{rank}\;A$ is odd then $\chi\left(A\right)=0$.
\end{thm}

The proofs of these results are very different to that of Theorem \ref{thm: euler}: the proof of Theorem \ref{thm: IKV} uses Mackenzie's spectral sequence \cite{MackenzieBook}, and the proof of Theorem \ref{thm: K} a version of Poincar\'{e} duality for Lie algebroids, see loc.~ cit.~ for the terminology.  We note that the result of \cite{ItskovKV98} holds if $M$ is noncompact but admits a finite good cover in the sense of \cite{BottTBook}. (The vanishing is not stated explicitly in \cite{ItskovKV98} but follows from the statement in loc.~ cit.~ and Theorem \ref{thm: Goldberg}.)

The following Theorem is a slight rephrasing of Theorem 3.1 in \cite{PflaumPT14}, which is an application of the higher index theorem for Lie groupoids proven in \cite{PflaumPT15}.
\begin{thm}
\label{thm: PPT}
\emph{(Pflaum, Posthuma \& Tang 2014, \cite{PflaumPT14}, Theorem 3.1.)}
If $A$ is integrable, oriented and unimodular then the index of the Euler operator $D_{A}$ is given by
\begin{equation}
\label{eqn: PPT}
\mathrm{Ind}_{\Omega}\left(D_{A}\right) = \int_{M}\left< e^{A}\left(A\right),\Omega \right>.
\end{equation}
\end{thm}
Here $\Omega$ is an invariant section of the vector bundle $\wedge^{\mathrm{top}}A\otimes \wedge^{\mathrm{top}}T^{*}M$ and $e^{A}\left(A\right) = a^{*} \left(e\left(A\right)\right)$ is the Lie algebroid Euler class of $A$, where $e\left(A\right)$ is the standard Euler class of $A$ and $a^{*}:H_{\mathrm{dR}}^{\bullet}\left(M\right) \to H^{\bullet}\left(A\right)$ is determined by $a:A\to TM$. See loc.~ cit.~ for further details. Under the assumptions of Theorem \ref{thm: PPT} one can deduce Theorem \ref{thm: euler} from \eqref{eqn: PPT}: if $A \ne TM$ is transitive then the left hand side can be identified with a non-zero multiple of $\chi\left(A\right)$, and $\mathrm{rank}\;A > \mathrm{dim}\;M$ implies that $e\left(A\right)$ and therefore $e^{A}\left(A\right)$ and the right hand side vanish.

We also mention \cite{Korman14}, where an index theorem is proved for certain non-transitive complex Lie algebroids called `elliptic involutive structures'.

\subsection{Organization of the paper}

In section \ref{sec: bg} we summarise the relevant definitions concerning Lie algebroids and their representations. The proofs of Theorems \ref{thm: euler} and \ref{thm: kunneth}, and of Corollaries \ref{cor: euler}, \ref{cor: kunneth} and \ref{cor: hopf} are in section \ref{sec: proofs}. In section \ref{sec: examples} we give several examples, including an example showing that Theorem \ref{thm: euler} does not hold for non-transitive Lie algebroids in general, and discuss the existence of compatible $H$-structures.

\subsection{Acknowledgements}

We would like to thank Xiang Tang for explaining Theorems \ref{thm: TYZ} and \ref{thm: PPT} and for helpful discussions about the results of this paper.

%%%%%%%%%%%%%%%%%%%%%%%%%
%%%%%%%%%%%%%%%%%%%%%%%%%

\section{Background}
\label{sec: bg}

We summarise the basic definitions regarding $C^{\infty}$ Lie algebroids and their representations. See \cite{MackenzieBook} for further details. Let $M$ be a smooth manifold. A \emph{Lie algebroid} over $M$ is a smooth vector bundle $A$ over $M$ equipped with an $\mathbb{R}$-linear Lie bracket on $\Gamma\left(A\right)$ and a vector bundle morphism $a:A \to TM$, called the anchor, such that the Leibniz rule
\[
\left[\xi,f\xi'\right] = \mathcal{L}_{a\left(\xi\right)}\left(f\right)\xi' + f\left[\xi,\xi'\right]
\]
holds for all $\xi,\xi' \in \Gamma\left(A\right)$ and $f \in C^{\infty}\left(M\right)$ where $\mathcal{L}$ denotes the Lie derivative. Complex Lie algebroids are defined similarly, replacing $TM$ by its complexification $T_{\mathbb{C}}M$. Standard examples include finite dimensional real or complex Lie algebras ($M=\mathrm{pt}$), the tangent bundle $TM$, and the \emph{Atiyah algebroid} $TP/G$ of a principal $G$-bundle $P\to M$ for $G$ a Lie group. These examples are all \emph{transitive}, meaning that the anchor map is surjective and there is a short exact sequence
\[
0 \to L \to A \to TM \to 0
\]
where $L=\mathrm{Ker}\;a$. In fact, $L$ is a locally trivial bundle of Lie algebras. See \cite{MackenzieBook} for the definition of a morphism of Lie algebroids.

Let $A$ be a Lie algebroid. Associated to $A$ is a cochain complex $\Gamma\left(\wedge^{\bullet}A^{*}\right)$ with differential $\mathrm{d}_{A}$ defined analogously to the de Rham differential by
\begin{align*}
\mathrm{d}_{A}f\left(\xi\right) & = \mathcal{L}_{a\left(\xi\right)}\left(f\right) \\
\mathrm{d}_{A}\omega\left(\xi,\xi'\right) & = \mathcal{L}_{a\left(\xi\right)}\left(\omega\left(\xi'\right)\right) - \mathcal{L}_{a\left(\xi'\right)}\left(\omega\left(\xi\right)\right) - \omega\left(\left[\xi,\xi'\right]\right)
\end{align*}
for $f \in C^{\infty}\left(M\right)$, $\omega\in\Gamma\left(A^{*}\right)$ and $\xi,\xi'\in\Gamma\left(A\right)$, and extended to $\Gamma\left(\wedge^{\bullet}A^{*}\right)$ by 
\begin{align*}
\mathrm{d}_{A}\left(\nu\wedge\nu'\right) = \mathrm{d}_{A}\nu\wedge\nu' + \left(-1\right)^{p}\nu\wedge\mathrm{d}_{A}\nu'
\end{align*}
for $\nu \in \Gamma\left(\wedge^{p}A^{*}\right)$ and $\nu'\in\Gamma\left(\wedge^{q}A^{*}\right)$. 

A \emph{representation} of $A$ consists of a smooth vector bundle $E$ over $M$ and a \emph{flat-$A$-connection}, which is a linear map $\nabla:\Gamma\left(E\right) \to \Gamma\left(E\otimes A^{*}\right)$ satisfying 
\begin{align}
\label{eqn: leibniz}
\nabla\left(fe\right) & = e\otimes \mathrm{d}_{A}f + f\nabla\left(e\right) 
\end{align}
and $\nabla^{2}=0$, for $f \in C^{\infty}\left(M\right)$ and $e \in \Gamma\left(E\right)$, where $\nabla$ is extended to $\Gamma\left(E\otimes \wedge^{p}A^{*}\right)$ by the rule 
\begin{equation}
\label{eqn: nabla}
\nabla\left(e\otimes \omega\right)=\nabla\left(e\right)\wedge \omega + e\otimes \mathrm{d}_{A}\omega.
\end{equation}  
The cohomology groups of the cochain complex $\Gamma\left(E\otimes \wedge^{\bullet}A^{*}\right)$ are denoted $H^{\bullet}\left(A,E\right)$, which coincides with the cohomology $H^{\bullet}\left(A\right)$ of $\Gamma\left(\wedge^{\bullet}A^{*}\right)$ if $E=M\times \mathbb{R}$ is the standard representation with $\nabla = \mathrm{d}_{A}$. Representations and cohomology of complex Lie algebroids are defined similarly. 

In the case that $M=\mathrm{pt}$ resp.~ $A=TM$ this reduces to Lie algebra cohomology resp. de Rham cohomology with flat vector bundle coefficients. If $A=TP/G$ is an Atiyah algebroid then there is an isomorphism $\Gamma\left(\wedge^{\bullet}A^{*}\right)\cong\Omega^{\bullet}\left(P\right)^{G}$.

The wedge product makes $\Gamma\left(\wedge^{\bullet}A^{*}\right)$ and $H^{\bullet}\left(A\right)$ into graded algebras, and a morphism of Lie algebroids $\phi:A \to B$ induces morphisms of graded algebras $\phi^{*}:\Gamma\left(\wedge^{\bullet}B^{*}\right) \to \Gamma\left(\wedge^{\bullet}A^{*}\right)$ and $\phi^{*}:H^{\bullet}\left(B\right)\to H^{\bullet}\left(A\right)$.

%%%%%%%%%%%%%%%%%%%%%%%%%
%%%%%%%%%%%%%%%%%%%%%%%%%

\section{Proofs of main results}
\label{sec: proofs}

\subsection{Proof of Theorem \ref{thm: euler}}
\label{sec: proofthm1}

It is shown in \cite{Krizka10} that for $x \in M$ and $\alpha \in T^{*}_{x}M$ the symbol complex of $\Gamma\left(E\otimes\wedge^{\bullet}A^{*}\right)$ at $\alpha$ is
\begin{equation}
\label{eqn: symbol}
\cdots \to E_{x} \otimes \wedge^{r}A_{x}^{*} \xrightarrow{\mathrm{id}\otimes\wedge a^{*}\left(\alpha\right)} E_{x}\otimes \wedge^{r+1}A_{x}^{*} \to \cdots 
\end{equation}
and is exact for non-zero $\alpha$ if $A$ is transitive. In particular, $\Gamma\left(E\otimes\wedge^{\bullet}A^{*}\right)$ is an elliptic complex and therefore the cohomology groups $H^{p}\left(A,E\right)$ are finite dimensional \cite{AtiyahB67}.

To calculate $\chi\left(A,E\right)$ we first reduce to a simpler case. Pulling back to the orientation double cover multiplies both $\chi\left(A,E\right)$ and the Euler characteristic $\chi\left(M\right)$ by $2$, complexification leaves $\chi\left(A,E\right)$ unchanged, and if $\mathrm{dim}\;M$ is odd then both the index of any elliptic complex and $\chi\left(M\right)$ are equal to $0$. We can therefore reduce to the case where $M$ is even dimensional and oriented and $A$ and $E$ are complex.

Let $\sigma$ denote the symbol class of the elliptic complex $\Gamma\left(E\otimes \wedge^{\bullet}A^{*}\right)$, $\sigma_{\mathrm{dR}}$ the symbol class of the complexified de Rham complex of $M$, $\pi:T^{*}M\to M$ the bundle projection, $\Psi$ the Thom isomorphism for $T^{*}M$, $e$ the Euler class of $TM$ and $\mathscr{T}$ the Todd class of $T_{\mathbb{C}}M$. Fix a splitting of $a:A \to T_{\mathbb{C}}M$. This determines isomorphisms $A \cong L\oplus T_{\mathbb{C}}M$
and 
\[
\wedge^{r} A^{*} \cong \bigoplus_{p+q=r} \wedge^{p}L^{*} \otimes \wedge^{q}T_{\mathbb{C}}^{*}M
\]
with respect to which the symbol complex \eqref{eqn: symbol} is
\[
\cdots \to \bigoplus_{p+q=r} \left(E_{x} \otimes \wedge^{p}L_{x}^{*}\right) \otimes \wedge^{q}T_{x}^{\ast}M  
 \xrightarrow{\mathrm{id} \otimes \wedge \alpha}
 \bigoplus_{p+q=r} \left(E_{x} \otimes \wedge^{p}L_{x}^{*}\right) \otimes \wedge^{q+1}T_{x}^{\ast}M \to \cdots 
\]
It follows that 
\begin{equation}
\label{eqn: sum}
\sigma = \sum_{p=0}^{\mathrm{rank}\;L}\left[\pi^{*}\left(E\otimes \wedge^{p}L^{*}\right)\right] \cdot \sigma_{\mathrm{dR}} \left[-p\right]
\end{equation}
where $\left[-p\right]$ denotes the shift of a complex by $p$. Using the fact that $\left[-1\right] = -\mathrm{Id}$ (see the Appendix of \cite{Segal68}), the naturality and multiplicativity of the Chern character and the fact that the Thom isomorphism is a morphism of $H^{\bullet}\left(M,\mathbb{Q}\right)$-modules we have
\begin{align}
\nonumber
\Psi^{-1}\mathrm{ch} \left(\sigma\right) & = \Psi^{-1}\mathrm{ch}\;\left(\sum_{p=0}^{\mathrm{rank}\;L}\left[\pi^{*}\left(E\otimes \wedge^{p}L^{*}\right)\right]\cdot \sigma_{\mathrm{dR}} \left[-p\right]  \right) \\
\nonumber
& = \Psi^{-1} \left( \sum_{p=0}^{\mathrm{rank}\;L}\left(-1\right)^{p}\pi^{*}\mathrm{ch}\left(E \otimes\wedge^{p}L^{*}\right) \cdot \mathrm{ch}\;\sigma_{\mathrm{dR}}\right) \\
\label{eqn: Thom}
& = \sum_{p=0}^{\mathrm{rank}\;L}\left(-1\right)^{p}\mathrm{ch}\left(E\otimes\wedge^{p}L^{*}\right)\cdot \Psi^{-1}\mathrm{ch}\;\sigma_{\mathrm{dR}}.
\end{align} 
Substituting \eqref{eqn: Thom} into the cohomological form of the Atiyah-Singer index theorem \cite{AtiyahS68b} and using the fact that $\Psi^{-1}\mathrm{ch}\;\sigma_{\mathrm{dR}}\cdot\mathscr{T}=e$ \cite{AtiyahS68b} is a top degree cohomology class gives
\begin{align*}
\chi\left(A,E\right) & = \left(\Psi^{-1}\mathrm{ch}\;\sigma \cdot \mathscr{T}\right)\left[M\right] \\
& = \left(\sum_{p=0}^{\mathrm{rank}\;L}\left(-1\right)^{p}\mathrm{ch}\left(E\otimes\wedge^{p}L^{*}\right)\cdot \Psi^{-1}\mathrm{ch}\;\sigma_{\mathrm{dR}} \cdot \mathscr{T}\right) \left[M\right] \\
%& = \left(\sum_{p=0}^{\mathrm{rank}\;L}\left(-1\right)^{p}\mathrm{ch}\left(E\otimes\wedge^{p}L^{*}\right)\cdot e\right) \left[M\right] \\
& = \left(\sum_{p=0}^{\mathrm{rank}\;L}\left(-1\right)^{p} \mathrm{rank}\;\wedge^{p}L^{*}\right) \mathrm{rank}\;E \cdot e\left[M\right] \\
& = \begin{cases}
\mathrm{rank}\;E \cdot \chi\left(M\right) \;\text{ if }\; L=0\\
0 \;\text{ otherwise }
\end{cases}
\end{align*}
where the last equality follows from the fact that the alternating sum of the binomial coefficients is zero. This completes the proof of Theorem \ref{thm: euler}.

\begin{remi}
Note that if one can make sense of dividing by the Euler class it is possible to use the equation
\[
\Psi^{-1} \mathrm{ch}\;\sigma \cdot e =
\sum_{p=0}^{\mathrm{rank}\;L}\left(-1\right)^{p} \mathrm{ch} \left(E\otimes \wedge^{p}A^{*}\right)
\]
to give a proof of Theorem \ref{thm: euler} involving only the vector bundles $E\otimes \wedge^{p}A^{*}$ and not the symbol $\sigma$.
\end{remi}

\begin{remi}
The map $A \mapsto \Gamma\left(\wedge^{\bullet}A^{*}\right)$ defines a 1-1 correspondence between transitive real Lie algebroids and real elliptic complexes of the form $\Gamma\left(\wedge^{\bullet}V\right)$ with differential a graded derivation. It follows that Theorem \ref{thm: euler} solves the index problem for every real elliptic complex of this type.
\end{remi}

\subsection{Proof of Corollary \ref{cor: euler}}

The Atiyah algebroid $TP/G$ of $P$ is a transitive Lie algebroid with $\mathrm{Ker}\;a \ne 0$. If $V$ is a non-zero finite dimensional real or complex representation of $G$ then the associated vector bundle $P\times_{G}V$ carries a natural flat $TP/G$-connection defined by $\nabla\left(v\right)\left(\xi\right):=\mathcal{L}_{\xi}\left(v\right)$, where $\Gamma\left(TP/G\right)$ is identified with $\Gamma\left(TP\right)^{G}$ and $\Gamma\left(P\times_{G}V\right)$ with $C^{\infty}\left(P,V\right)^{G}$. It is shown in Proposition 5.3.11 in \cite{MackenzieBook} that there is an isomorphism of cochain complexes 
\begin{equation*}
\Gamma\left(P\times_{G} V\otimes \wedge^{\bullet}\left(TP/G\right)^{*}\right) \cong \Omega^{\bullet}\left(P,V\right)^{G}.
\end{equation*}
The first two statements of Corollary \ref{cor: euler} then follow from Theorem \ref{thm: euler} and the third from the fact that if $G$ is compact then the inclusion $\Omega^{\bullet}\left(P,V\right)^{G} \hookrightarrow \Omega^{\bullet}\left(P,V\right)$ induces an isomorphism of cohomology groups \cite{GreubHVBookII}.

Specialising to the standard one dimensional representation proves the claims for $H_{\mathrm{dR},G}^{\bullet}\left(P\right)$, $\chi\left(P,G\right)$ and $\chi\left(P\right)$. This completes the proof of Corollary \ref{cor: euler}.

\subsection{Proof of Theorem \ref{thm: kunneth}}
\label{sec: kunnethproof}

We will show that there is an isomorphism 
\[
\Gamma\left(E\boxtimes F\otimes \wedge^{\bullet}\left(A\boxplus B\right)^{*}\right) 
\cong
\Gamma\left(E\otimes \wedge^{\bullet}A^{*}\right) \boxtimes \Gamma\left(F\otimes\wedge^{\bullet}B^{*}\right)
\]
where the right hand side is the outer tensor product of elliptic complexes. The first statment then follows from the K\"{u}nneth theorem for elliptic complexes stated in \cite{AtiyahB67}; see also Theorem 1.3 in \cite{Tarkhanov86} and %Theorem 1.4.8 
section 1.4.3 in \cite{TarkhanovBook}. See \cite{MackenzieBook} for products of Lie algebroids and \cite{Crainic03} for pullbacks and tensor products of representations.

Denote by  $\pM:M\times N \to M$ and $\pN : M\times N \to N$ the two projections. If $\xi\in\Gamma\left(A\right)$ and $\nu\in\Gamma\left(B\right)$ then we set $\xi\boxplus\nu := \left(\mathrm{pr}_{M}^{*}\xi,\mathrm{pr}_{N}^{*}\nu\right) \in \Gamma\left(A\boxplus B\right)$. We use a similar notation $e\boxtimes f:=\mathrm{pr}_{M}^{*}e \otimes \mathrm{pr}_{N}^{*}$ for sections of $E\boxtimes F$ and other outer tensor products of vector bundles.

The anchor map of $A\boxplus B$ is given by the direct sum of the anchor maps of $A$ and $B$, and the Lie bracket is determined by the Leibniz rule and the definition 
\[
\left[\xi \boxplus \nu,\xi'\boxplus \nu' \right] = \left[\xi,\xi'\right] \boxplus \left[\nu,\nu'\right].
\]
This implies that with respect to the canonical isomorphisms
\begin{equation*}
\wedge^{r}\left(A \boxplus B\right)^{*}\cong\bigoplus_{p+q=r}  \wedge^{p} A^{*} \boxtimes \wedge^{q} B^{*}
\end{equation*}
the differential $\mathrm{d}_{A\times B}$ on $\Gamma\left(\wedge^{\bullet}\left(A\boxplus B\right)^{*}\right)$ is given by
\begin{equation}
\label{eqn: dAB}
\mathrm{d}_{A\times B} \left(\omega \boxtimes \delta\right) = \mathrm{d}_{A}\omega \boxtimes \delta + \left(-1\right)^{p} \omega \boxtimes \mathrm{d}_{B}\delta
\end{equation}
for $\omega \in \Gamma\left(\wedge^{p}A^{*}\right)$ and $\delta \in \Gamma\left(\wedge^{q}B^{*}\right)$.

There are flat $A\times B$ connections $\nabla^{E}$ on $\pM E$ and $\nabla^{F}$ on $\pN F$ defined via the natural Lie algebroid morphisms $A\times B \to A$ and $A\times B \to B$. As a representation of $A\times B$, $E\boxtimes F$ is by definition the tensor product of the representations $\pM E$ and $\pN F$. Explicitly, the flat connection $\nabla^{E\boxtimes F}$ on $E\boxtimes F$ is determined by the Leibniz rule \eqref{eqn: leibniz} and
\begin{equation}
\label{eqn: nablaEF}
\nabla^{E\boxtimes F}\left(e\boxtimes f\right) = \nabla^{E}\left(e\right)\boxtimes f + e\boxtimes\nabla^{F}\left(f\right),
\end{equation}
where the terms on the right hand side are defined via the identification of $\left(E\otimes A^{*}\right) \boxtimes F$ and $E \boxtimes \left(F \otimes B^{*}\right)$ with subbundles of $E\boxtimes F \otimes \left(A^{*} \boxplus B^{*}\right)$. It then follows from \eqref{eqn: nabla}, \eqref{eqn: dAB} and \eqref{eqn: nablaEF} that with respect to the canonical isomorphisms
\[
E\boxtimes F \otimes \wedge^{r}\left(A\times B\right)^{*} \cong 
\bigoplus_{p+q=r} \left(E \otimes \wedge^{p}A^{*}\right) \boxtimes \left(E \otimes \wedge^{q}B^{*}\right)
\]
the extension of $\nabla^{E\boxtimes F}$ to higher exterior powers of $\left(A\times B\right)^{*}$ is 
\begin{align*}
\nabla^{E\boxtimes F} \left(\left(e \otimes \omega\right) \boxtimes \left(f \otimes \delta\right) \right) 
& =  \nabla^{E}\left(e \otimes \omega\right) \boxtimes \left(f\otimes \delta\right) \\
& +   \left(-1\right)^{p}  \left(e \otimes \omega\right) \boxtimes \nabla^{F}\left(f\otimes \delta\right)
\end{align*}
for $e\in \Gamma\left(E\right)$, $f\in\Gamma\left(F\right)$, $\omega\in\Gamma\left(\wedge^{p}A^{*}\right)$ and $\delta\in\Gamma\left(\wedge^{q}B^{*}\right)$ as required. 

If $E=M\times\mathbb{R}$ and $F=N\times\mathbb{R}$ are the standard representations then the maps $H^{\bullet}\left(A\right) \to H^{\bullet}\left(A\times B\right)$ and $H^{\bullet}\left(A\right) \to H^{\bullet}\left(A\times B\right)$ determined by the projections $A\times B \to A$ and $A\times B \to B$ are morphisms of graded algebras and therefore so is the isomorphism of graded vector spaces $H^{\bullet}\left(A\right)\otimes H^{\bullet}\left(B\right) \to H^{\bullet}\left(A \times B\right)$, $\left[\omega\right]\otimes \left[\delta\right] \mapsto \left[\omega\boxtimes\delta\right]$. This completes the proof of Theorem \ref{thm: kunneth}.

\begin{remi}
Theorem \ref{thm: kunneth} continues to hold if $M$ and $N$ are noncompact but the cohomology groups $H^{p}\left(A,E\right)$ and $H^{q}\left(B,F\right)$  are finite dimensional, in which case they are Hausdorff topological vector spaces and Theorem 1.3 of \cite{Tarkhanov86} still applies.
\end{remi}

\subsection{Proof of Corollary \ref{cor: kunneth}}

The canonical isomorphism $TP\times TQ \cong T\left(P\times Q\right)$ of Lie algebroids is $G\times H$ equivariant and therefore descends to an isomorphism $TP/G \times TQ/H \cong T\left(P\times Q\right)/\left(G\times H\right)$. The statement then follows from Theorem \ref{thm: kunneth}. This completes the proof of Corollary \ref{cor: kunneth}.

\subsection{Proof of Corollary \ref{cor: hopf}}

The proof of the first statement follows the proof of Hopf's theorem on the structure of the cohomology ring of an $H$-space given in section 3C of \cite{HatcherBook}. (Note that in \cite{HatcherBook} the term `Hopf algebra' is used to describe a structure closer to that of a bialgebra, see Remark 20.3.2 in \cite{MayPBookConcise} for a discussion of this point and chapter 20 of loc.~cit.~ for the definitions of bialgebras and Hopf algebras.) 

Define $\Delta$ to be the composition
\[
H^{\bullet}\left(A\right) \xrightarrow{H^{*}} H^{\bullet}\left(A\times A\right) \xrightarrow{\cong} H^{\bullet}\left(A\right) \otimes H^{\bullet}\left(A\right)
\]
where $H^{*}$ is the map on cohomology determined by the Lie algebroid morphism $H:A\times A \to A$ and the second map is the K\"{u}nneth isomorphism of Theorem \ref{thm: kunneth}. The map $\Delta$ is a graded algebra morphism as it is a composition of such maps. 

The projection $\varepsilon : H^{\bullet}\left(A\right) \to H^{0}\left(A\right)$ is an algebra morphism. 
Define 
\[
H^{+}\left(A\right) := \mathrm{Ker}\;\varepsilon = \bigoplus_{p=1}^{\mathrm{rank}\;A} H^{p}\left(A\right).
\]
The canonical algebra homomorphism $\mathbb{R} \to H^{0}\left(A\right)$ mapping $\lambda \in \mathbb{R}$ to the corresponding constant function is an isomorphism: if $f \in C^{\infty}\left(M\right)$ and $\left(\mathrm{d}_{A}f\right)\left(\xi\right) = \mathcal{L}_{a\left(\xi\right)}\left(f\right) = 0$ for all $\xi \in \Gamma\left(A\right)$ then $f$ is constant because $a:A\to TM$ is surjective and $M$ is connected.

As $e$ is contained in the zero section of $A$ the map $A \to A$, $x \mapsto e$ is a Lie algebroid morphism. It then follows from the fact that $A \times A$ is a product in the category of Lie algebroids \cite{MackenzieBook} that the map $A \to A\times A$, $x \mapsto \left(x,e\right)$ is a Lie algebroid morphism also. The same argument as in the $H$-space case (see the diagram on p.~283 of \cite{HatcherBook}) then shows that 
\[
\Delta\left(\omega\right) - \omega\otimes 1 - 1\otimes \omega \in H^{+}\left(A\right) \otimes H^{+}\left(A\right)
\]
for $\omega \in H^{p}\left(A\right)$ with $p>0$.

The proceeding discussion shows that $H^{\bullet}\left(A\right)$ and $\Delta$ satisfy the assumptions in the algebraic form of Hopf's theorem \cite{Hopf41}, see Theorem 3C.4 in \cite{HatcherBook} or section 2.4 in \cite{Cartier07}, which shows that $H^{\bullet}\left(A\right)$ is an exterior algebra with generators of odd degree.

Now assume that $H$ is associative. Then $\Delta$ is coassociative and $H^{\bullet}\left(A\right)$ together with $\Delta$ and $\varepsilon:H^{\bullet}\left(A\right) \to H^{0}\left(A\right) \cong \mathbb{R}$ is a bialgebra. It is straightforward to check that $H^{\bullet}\left(A\right)$ satisfies Definition 21.3.1 of \cite{MayPBookConcise}, and then Proposition 21.3.3 in loc.~ cit~ shows that $H^{\bullet}\left(A\right)$ admits a unique antipode and is therefore a Hopf algebra.

\section{Examples and further results}
\label{sec: examples}

\subsection{Action Lie algebroids}

Let $\g$ be a finite dimensional real Lie algebra, $M$ a smooth manifold and $\phi:\g \to \Gamma\left(TM\right)$ a Lie algebra homomorphism. The Lie derivative then makes $C^{\infty}\left(M\right)$ into a $\g$ module, and by evaluation, $\phi$ determines a linear map $\phi_{x}:\g \to T_{x}M$ for each $x\in M$.
\begin{prop}
Assume that $M$ is compact and $\phi_{x}$ is surjective for all $x\in M$.
\begin{enumerate}
\item The Lie algebra cohomology groups $H^{p}\left(\g,C^{\infty}\left(M\right)\right)$ are finite dimensional. 
\item $
%\chi\left(\g,C^{\infty}\left(M\right)\right) := 
\sum_{p=0}^{\mathrm{rank}\;\g} \mathrm{dim}\;H^{p}\left(\g,C^{\infty}\left(M\right)\right) =  
\begin{cases}
\chi\left(M\right) \;\text{ if }\; \mathrm{dim}\;\g = \mathrm{dim}\; M \\
0 \;\text{ else. }\;
\end{cases}
$
\end{enumerate}
\end{prop}

\begin{proof}
Associated to $\phi:\g\to \Gamma\left(TM\right)$ is the \emph{action Lie algebroid} $\g\ltimes M$, for which the complex $\Gamma\left(\wedge^{\bullet}\left(\g\ltimes M\right)^{*}\right)$ is isomorphic to the Chevalley-Eilenberg complex $\wedge^{\bullet}\g^{*} \otimes C^{\infty}\left(M\right)$ \cite{MackenzieBook}. Under the assumption on $\phi$ the action Lie algebroid is transitive and the result follows from Theorem \ref{thm: euler}.
\end{proof}

\subsection{Non-transitive Lie algebroids}

The following Example shows that if $A$ is not transitive and $H^{p}\left(A,E\right)$ is finite dimensional then $\chi\left(A,E\right)$ is in general non-zero. 
\begin{ex}
Let $M=\mathbb{R}$, $n \in \mathbb{N}$ and $p$ the polynomial function $\left(t-1\right)\cdots\left(t-n\right)$. Consider the Lie algebroid $A:=M\times \mathbb{R}$ with anchor map $f \mapsto p\partial_{t}$ and Lie bracket
\[
\left[f,g\right] = p\left(f\frac{dg}{dt} -g\frac{df}{dt}\right).
\]
The complex $\Gamma\left(\wedge^{\bullet}A^{*}\right)$ is isomorphic to the non-elliptic complex
\[
C^{\infty}\left(\mathbb{R}\right) \xrightarrow{p\partial_{t}} C^{\infty}\left(\mathbb{R}\right)
\]
which has cohomology groups $\mathrm{Ker}\left(p\partial_{t}\right)=\mathbb{R}$ and $\mathrm{Coker}\left(p\partial_{t}\right)\cong \mathbb{R}^{n}$ because $\partial_{t}$ is surjective and $p$ generates the vanishing
 ideal of $p^{-1}\left(0\right)$. In particular $\chi\left(A\right) = 1-n$.
\end{ex}

\begin{remi}
One can give an analogous example with $M$ compact by replacing $\mathbb{R}$ by $S^{1}$ and $p$ by any smooth function with $n$ isolated zeros of order 1. In this case $\chi\left(A\right)=-n$ which is not a multiple of $\chi\left(S^{1}\right)=0$.
\end{remi}

\subsection{$H$-structures}
\label{sec: examplesH}

Throughout section \ref{sec: examplesH} $A$ denotes a Lie algebroid over $M$ equipped with a compatible $H$-structure $H:A\times A \to A$ and $e \in A$ is a unit for $H$ (see the paragraph above Corollary \ref{cor: hopf} for the definition). If $H$ covers a smooth map $f:M\times M \to M$ and $e \in A_{m}$ then $m$ is a unit for $f$ and $M$ is an $H$-space. The following topological restrictions on $H$-spaces are well known, see section 3.C of \cite{HatcherBook}.
\begin{prop}
If $M$ is connected then $\pi_{1}\left(M\right)$ is abelian. If $M$ is also compact and positive dimensional then $\chi\left(M\right)=0$ and $H^{\bullet}\left(M,\mathbb{Q}\right)$ is an exterior algebra.
\end{prop}

\begin{prop}
\label{prop: Hg}
Let $\g$ be a Lie algebra. There exists a Lie algebra morphism $H:\g \times \g \to \g$ satisfying $H\left(0,x\right) = H\left(x,0\right) = x$ for all $x\in \g$ if and only if $\g$ is abelian, in which case the addition map $\left(x,y\right) \mapsto x+y$ is the unique map satisfying these conditions.
\end{prop}

\begin{proof}
Suppose that $H:\g \times \g \to \g$ is a linear map and write $H\left(x,y\right) = H_{1}\left(x\right) + H_{2}\left(y\right)$. Then $H$ is a Lie algebra morphism if and only if $H_{1}$ and $H_{2}$ are Lie algebra morphisms from $\g$ to $\g$ whose images commute. The condition $H\left(0,x\right) = H\left(x,0\right) = x$ is equivalent to $H_{2}\left(x\right) = H_{1}\left(x\right)= x$ and therefore $H_{1}=H_{2}=\mathrm{id}_{\g}$. As the images of $H_{1}$ and $H_{2}$ commute we must have that $\g$ is abelian.
\end{proof}

\begin{remi}
If $\g$ is abelian then $H^{\bullet}\left(\g\right) = \wedge^{\bullet}\g^{*}$ and if $\g$ is also finite dimensional then the Hopf algebra structure associated to the unique $H$-structure by Corollary \ref{cor: hopf} is the standard Hopf algebra structure on an exterior algebra.
\end{remi}

\begin{prop}
If $M$ is connected then the fibres of $L=\mathrm{Ker}\; a$ are abelian.
\end{prop}

\begin{proof}
As $L$ is a locally trivial bundle of Lie algebras \cite{MackenzieBook} it is sufficient to show that $L_{m}$ is abelian.
$H$ restricts to a linear map $A_{m}\times A_{m} \to A_{m}$, and to a morphism of Lie algebras $L_{m}\times L_{m} \to L_{m}$ because $H$ is a morphism of transitive Lie algebroids. Applying Proposition \ref{prop: Hg} to this morphism shows that $L_{m}$ is abelian.
\end{proof}

%\cite{*}
 
\bibliography{mybib}{}

\begin{thebibliography}{10}

\bibitem{AtiyahB67}
{\sc M.~F. Atiyah and R.~Bott}, {\em A {L}efschetz fixed point formula for
  elliptic complexes. {I}}, Ann. of Math. (2), 86 (1967), pp.~374--407.

\bibitem{AtiyahS68b}
{\sc M.~F. Atiyah and I.~M. Singer}, {\em The index of elliptic operators.
  {III}}, Ann. of Math. (2), 87 (1968), pp.~546--604.

\bibitem{BottTBook}
{\sc R.~Bott and L.~W. Tu}, {\em Differential forms in algebraic topology},
  vol.~82 of Graduate Texts in Mathematics, Springer-Verlag, New York-Berlin,
  1982.

\bibitem{Cartier07}
{\sc P.~Cartier}, {\em A primer of {H}opf algebras}, in Frontiers in number
  theory, physics, and geometry. {II}, Springer, Berlin, 2007, pp.~537--615.

\bibitem{ChevalleyE48}
{\sc C.~Chevalley and S.~Eilenberg}, {\em Cohomology theory of {L}ie groups and
  {L}ie algebras}, Trans. Amer. Math. Soc., 63 (1948), pp.~85--124.

\bibitem{Crainic03}
{\sc M.~Crainic}, {\em Differentiable and algebroid cohomology, van {E}st
  isomorphisms, and characteristic classes}, Comment. Math. Helv., 78 (2003),
  pp.~681--721.

\bibitem{Goldberg55}
{\sc S.~I. Goldberg}, {\em On the {E}uler characteristic of a {L}ie algebra},
  Amer. Math. Monthly, 62 (1955), pp.~239--240.

\bibitem{GreubHVBookII}
{\sc W.~Greub, S.~Halperin, and R.~Vanstone}, {\em Connections, curvature, and
  cohomology. {V}ol. {II}: {L}ie groups, principal bundles, and characteristic
  classes}, Academic Press [A subsidiary of Harcourt Brace Jovanovich,
  Publishers], New York-London, 1973.
\newblock Pure and Applied Mathematics, Vol. 47-II.

\bibitem{HatcherBook}
{\sc A.~Hatcher}, {\em Algebraic topology}, Cambridge University Press,
  Cambridge, 2002.

\bibitem{Hopf41}
{\sc H.~Hopf}, {\em \"{U}ber die {T}opologie der {G}ruppen-{M}annigfaltigkeiten
  und ihre {V}erallgemeinerungen}, Ann. of Math. (2), 42 (1941), pp.~22--52.

\bibitem{ItskovKV98}
{\sc V.~Itskov, M.~Karasev, and Y.~Vorobjev}, {\em Infinitesimal {P}oisson
  cohomology}, in Coherent transform, quantization, and {P}oisson geometry,
  vol.~187 of Amer. Math. Soc. Transl. Ser. 2, Amer. Math. Soc., Providence,
  RI, 1998, pp.~327--360.

\bibitem{Korman14}
{\sc E.~O. Korman}, {\em Elliptic Involutive Structures and Generalized Higgs
  Algebroids}, PhD thesis, University of Pennsylvania, 2014.
\newblock Available at
  \url{https://web.ma.utexas.edu/users/ekorman/files/dissertation.pdf}.

\bibitem{Kubarski02}
{\sc J.~Kubarski}, {\em Poincar\'{e} duality for transitive unimodular
  invariantly oriented {L}ie algebroids}, Topology Appl., 121 (2002),
  pp.~333--355.

\bibitem{Krizka10}
{\sc L.~K\v{r}i\v{z}ka}, {\em Moduli spaces of flat lie algebroid connections}.
\newblock Preprint, \url{https://arxiv.org/abs/1012.3180}, 2010.

\bibitem{Mackenzie92}
{\sc K.~C.~H. Mackenzie}, {\em Double {L}ie algebroids and second-order
  geometry. {I}}, Adv. Math., 94 (1992), pp.~180--239.

\bibitem{MackenzieBook}
\leavevmode\vrule height 2pt depth -1.6pt width 23pt, {\em General theory of
  {L}ie groupoids and {L}ie algebroids}, vol.~213 of London Mathematical
  Society Lecture Note Series, Cambridge University Press, Cambridge, 2005.

\bibitem{MayPBookConcise}
{\sc J.~P. May and K.~Ponto}, {\em More concise algebraic topology}, Chicago
  Lectures in Mathematics, University of Chicago Press, Chicago, IL, 2012.
\newblock Localization, completion, and model categories.

\bibitem{PflaumPT14}
{\sc M.~J. Pflaum, H.~Posthuma, and X.~Tang}, {\em The index of geometric
  operators on {L}ie groupoids}, Indag. Math. (N.S.), 25 (2014),
  pp.~1135--1153.

\bibitem{PflaumPT15}
\leavevmode\vrule height 2pt depth -1.6pt width 23pt, {\em The localized
  longitudinal index theorem for {L}ie groupoids and the van {E}st map}, Adv.
  Math., 270 (2015), pp.~223--262.

\bibitem{Segal68}
{\sc G.~Segal}, {\em Equivariant {$K$}-theory}, Inst. Hautes \'{E}tudes Sci.
  Publ. Math.,  (1968), pp.~129--151.

\bibitem{Serre51}
{\sc J.-P. Serre}, {\em Homologie singuli\`ere des espaces fibr\'{e}s.
  {A}pplications}, Ann. of Math. (2), 54 (1951), pp.~425--505.

\bibitem{TangYZ13}
{\sc X.~Tang, Y.-J. Yao, and W.~Zhang}, {\em Hopf cyclic cohomology and {H}odge
  theory for proper actions}, J. Noncommut. Geom., 7 (2013), pp.~885--905.

\bibitem{Tarkhanov86}
{\sc N.~N. Tarkhanov}, {\em Alexander duality for elliptic complexes}, Mat. Sb.
  (N.S.), 130(172) (1986), pp.~62--85, 128.

\bibitem{TarkhanovBook}
{\sc N.~N. Tarkhanov}, {\em Complexes of differential operators}, vol.~340 of
  Mathematics and its Applications, Kluwer Academic Publishers Group,
  Dordrecht, 1995.
\newblock Translated from the 1990 Russian original by P. M. Gauthier and
  revised by the author.

\bibitem{Zusmanovich11}
{\sc P.~Zusmanovich}, {\em How {E}uler would compute the {E}uler-{P}oincar\'{e}
  characteristic of a {L}ie superalgebra}, Expo. Math., 29 (2011),
  pp.~345--360.

\end{thebibliography}
\bibliographystyle{siam}

\vspace{0.5cm}
\noindent James Waldron, School of Mathematics, Statistics and Physics, Newcastle University, Newcastle upon Tyne NE1 7RU, UK.
\newline
\noindent Email address: james.waldron@newcastle.ac.uk

\end{document}